\documentclass[a4paper,twoside,11pt]{article}
\usepackage{amssymb,amsmath,amsthm,latexsym}
\usepackage{amsfonts}
\usepackage{graphicx,caption,subcaption}
\usepackage[pdftex,bookmarks,colorlinks=false]{hyperref}
\usepackage[a4paper, total={14cm, 22cm}]{geometry}

\usepackage[auth-sc-lg,affil-sl]{authblk}
\newtheorem{thm}{Theorem}[section]

\newtheorem{defn}[thm]{Definition}

\newtheorem{lem}[thm]{Lemma}

\def\ni{\noindent}
\def\dmd{\diamond}

\title{\textbf{\sc On the Sparing Number of the Edge-Corona of Graphs}}

\author{K. P. Chithra}
\affil{\small Naduvath Mana, Nandikkara\\ Thrissurd-680301, Kerala, India.\\ E-mail: chithrasudev@gmail.com}

\author{K. A. Germina}
\affil{\small PG \& Research Department of Mathematics\\ Mary Matha Arts \& Science College\\Mananthavady, Wayanad-670645, Kerala, India.\\ E-mail: srgerminaka@gmail.com}

\author{N.K. Sudev}
\affil{\small Department of Mathematics\\ Vidya Academy of Science \& Technology \\ Thalakkottukara, Thrissur - 680501, Kerala, India.\\ E-mail: sudevnk@gmail.com}

\date{}

\begin{document}
\maketitle

\begin{abstract}
Let $\mathbb{N}_0$ be the set of all non-negative integers and $\mathcal{P}(\mathbb{N}_0)$ be its the power set. An integer additive set-indexer (IASI) of a graph $G$ is an injective function $f:V(G)\to \mathcal{P}(\mathbb{N}_0)$ such that the induced function $f^+:E(G) \to \mathcal{P}(\mathbb{N}_0)$ defined by $f^+ (uv) = f(u)+ f(v)$ is also injective, where $f(u)+f(v)$ is the sum set of $f(u)$ and $f(v)$. An integer additive set-indexer $f$ is said to be a weak integer additive set-indexer (weak IASI) if $|f^+(uv)|=\max(|f(u)|,|f(v)|)~\forall ~ uv\in E(G)$. The minimum number of singleton set-labeled edges required for the graph $G$ to admit an IASI is called the sparing number of the graph.  In this paper, we discuss the admissibility of weak IASI by a particular type of graph product called the edge corona of two given graphs and determine the sparing number of the edge corona of certain graphs.
\end{abstract}

\ni \textbf{Key Words:} Integer additive set-indexers, mono-indexed elements of a graphs, weak integer additive set-indexers, sparing number of a graph, edge corona of a graph.

\section{Introduction}

For all  terms and definitions, not defined specifically in this paper, we refer to \cite{BM}, \cite{FH} and \cite{DBW}. Unless mentioned otherwise, all graphs considered here are simple, finite and have no isolated vertices.

The sum set of two sets $A$ and $B$, denoted $A+B$, is the set defined by $A + B = \{a+b: a \in A, b \in B\}$. If either $A$ or $B$ is countably infinite, then their sum set will also be countably infinite. Hence, all sets we consider in this study are finite sets. The cardinality of a set $A$ is denoted by $|A|$. The power set of  a set $A$ is denoted by $\mathcal{P}(A)$.
 
Let $\mathbb{N}_0$ denote the set of all non-negative integers. An {\em integer additive set-indexer} (IASI, in short) of a graph $G$ is defined in \cite{GA} as an injective function $f:V(G)\to \mathcal{P}(\mathbb{N}_0)$ such that the induced function $f^+:E(G) \to \mathcal{P}(\mathbb{N}_0)$ defined by $f^+ (uv) = f(u)+ f(v)$ is also injective.

A \textit{weak IASI} $f$ is (see \cite{GS1}) an IASI such that $|f^+(uv)|= \max(|f(u)|,|f(v)|)$ for all $u,v\in V(G)$. A weak IASI $f$ is said to be {\em weakly uniform IASI} if $|f^+(uv)|=k$, for all $u,v\in V(G)$ and for some positive integer $k$.  A graph which admits a weak IASI may be called a {\em weak IASI graph}. 

The following result is a necessary and sufficient condition for a given graph to admit a weak IASI. 

\begin{lem}
\cite{GS1} A graph $G$ admits a weak integer additive set-indexer if and only if every edge of $G$ has at least one mono-indexed end vertex.
\end{lem}

The following definitions are made in \cite{GS3}. The cardinality of the labeling set of an element (vertex or edge) of a graph $G$ is called the {\em set-indexing number} of that element. An element (a vertex or an edge) of graph which has the set-indexing number $1$ is called a {\em mono-indexed element} of that graph. The {\em sparing number} of a graph $G$ is defined to be the minimum number of mono-indexed edges required for $G$ to admit a weak IASI and is denoted by $\varphi(G)$.

\ni Certain Studies about weak IASIs have been done already and the following are some major results about weak IASI graphs relevant in this study.

\begin{thm}\label{T-WSG}
\cite{GS3} A subgraph of weak IASI graph is also a weak IASI graph.
\end{thm}

\begin{thm}\label{T-WUC}
\cite{GS3} A graph $G$ admits a weak IASI if and only if $G$ is bipartite or it has at least one mono-indexed edge. Also, the sparing number of a bipartite graph $G$ is $0$.
\end{thm}

\begin{thm}\label{T-WUOC}
\cite{GS3} Let $C_n$ be a cycle of length $n$ which admits a weak IASI, for a positive integer $n$. Then, $C_n$ has an odd number of mono-indexed edges when it is an odd cycle and has even number of mono-indexed edges, when it is an even cycle.  An odd cycle $C_n$ has a weak IASI if and only if it has at least one mono-indexed edge. 
\end{thm}

\begin{thm}\label{T-WUG}
\cite{GS4} The graph $G_1\cup G_2$ admits a weak IASI if and only if both $G_1$ and $G_2$ are weak IASI graphs. More over, $\varphi(G_1 \cup G_2)=\varphi(G_1)+\varphi(G_2)-\varphi(G_1 \cap G_2)$.
\end{thm}

\begin{thm}\label{T-WKN}
\cite{GS3} A complete graph can have at most one vertex that is not mono-indexed. Also, the sparing number of a complete graph $K_n$ is $\frac{1}{2}(n-1)(n-2)$.
\end{thm}

The admissibility of weak IASI by certain graph products and their sparing numbers have been studied in \cite{CGS1}, \cite{CGS2} and \cite{GS7}.  In this paper, our intention is to study about the admissibility of weak IASI by a particular product, called edge corona, of two given graphs and estimate the corresponding sparing number. 
 
\section{The Sparing Number of Edge Corona of Graphs}

\ni Let us first recall the definition of the edge corona of two graphs.

\begin{defn}\label{D-5.1}{\rm
\cite{HS} Let $G_1$ be a graph with $n_1$ vertices and $m_1$ edges and $G_2$ be a graph with $n_2$ vertices and $m_2$ edges. Then, the {\em edge corona} of $G_1$ and $G_2$, denoted by $G_1\dmd G_2$, is the graph obtained by taking $m_1$ copies of $G_2$ and then joining the end vertices of $i$-th edge of $G_1$ to every vertex in the $i$-th copy of $G_2$.}
\end{defn}

Figure \ref{fig:G-ECor} is an example for the graph which is the edge corona of the cycles $C_5$ and $C_3$.

\begin{figure}[h!]
\centering
\includegraphics[width=0.6\linewidth]{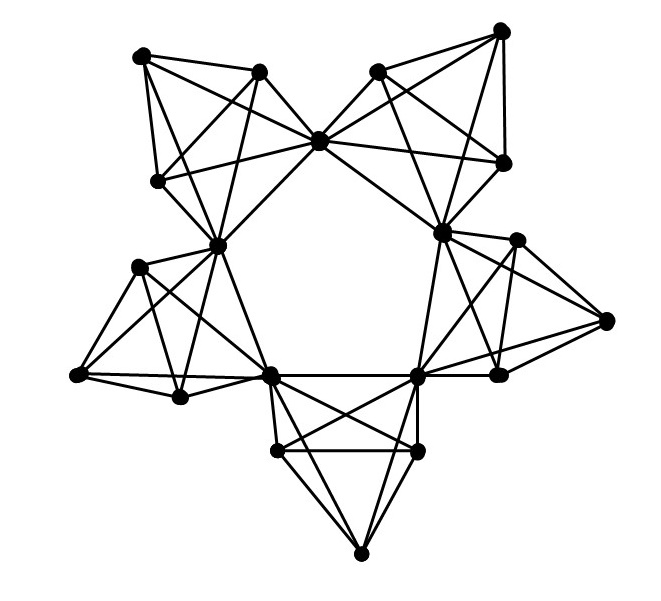}
\caption{The edge corona $C_5\dmd C_3$.}
\label{fig:G-ECor}
\end{figure}

The weak IASIs of $G_1$ and $G_2$ may not induce a weak IASI for $G_1\dmd G_2$. Hence, we need to define an IASI independently for a graph product. 

We say that a graph $G$ is said to be a \textit{$1$-uniform graph} if the set-labels of all elements (vertices and edges) of $G$ are singleton sets. By the term an \textit{integral multiple of a set $A$}, we mean the set obtained by multiplying every element of $A$ by a same integer.

The following theorem establishes a necessary condition for the edge corona of two  weak IASI graphs to admit a weak IASI.

\begin{thm}\label{T-SNECG1}
For two given graphs $G_1$ and $G_2$, if $G_1\dmd G_2$ admits a weak IASI, then either $G_1$ is $1$-uniform or $m_1-m_1'$ copies of $G_2$ are $1$-uniform, where $m_1'$ is the number of mono-indexed edges in $G_1$.
\end{thm}
\begin{proof}
Let $G_1$ be a graph on $n_1$ vertices and and $m_1$ edges and $G_2$ be a graph on $n_2$ vertices and and $m_2$ edges. Let $V(G_1)=\{v_1,v_2,\ldots v_{n_1}\}$ and $V(G_2)=\{u_1,u_2,\ldots, u_{n_2}\}$ be the vertex sets and $E(G_1)=\{e_1,e_2,\ldots, e_{m_1}\}$ and $E(G_2)=\{e'_1,e'_2,\ldots, e'_{m_2}\}$ be the edge sets of $G_1$ and $G_2$ respectively. Let $G_{2,j}$ be the $j$-th copy of $G_2$ corresponding to the $j$-th edge $e_j=v_rv_s$ of $G_1$ in $G_1\dmd G_2$ and $V(G_{2,j})=\{u_{1j},u_{2j},\ldots,u_{n_2j} \}$.  Then, the subgraph of $G_1\dmd G_2$ induced by the vertices $\{v_r,v_s, u_{kj}, u_{lj}\}$ is the complete graph $K_4$, for any two adjacent vertices $u_{kj}$ and $u_{lj}$ in $G_{2,j}$. That is, all edges of $G_{2,i}$ are the edges of different complete graphs $K_4$ in $G_1\dmd G_2$, all of these complete graphs have the common edge $e_j=v_rv_s$.

First assume that $G_1\dmd G_2$ admits a weak IASI. Then, we have to consider the following two cases.

\ni {Case-1:} Assume that $G_1$ is not $1$-uniform. Then, $G_1$ will have some elements which are not mono-indexed. Without loss of generality, assume that the edge $e_j$ is not mono-indexed. Then either $v_r$ or $v_s$ must have a non-singleton set-label. Let $v_r$ be the vertex that is not mono-indexed. Then, by Theorem \ref{T-WKN}, no other vertex $v_{lj}$ can have a non-singleton set-label. Therefore, the copy $G_{2,j}$ is $1$-uniform. This argument is valid for the copies of $G_2$ corresponding to all edges of $G$ that are not mono-indexed. Therefore, at least $m_1-m_1'$ copies of $G_2$ must be $1$-uniform, where $m_1'$ is the number of mono-indexed edges in $G_1$.

\ni {Case-2:} Assume that no copy of $G_2$ is $1$-uniform. Then, each copy $G_{2,j}$ of $G_2$ has at least one edge that is not mono-indexed. Let the edge $u_{kj}u_{lj}$ of $G_{2,j}$ has the non-singleton set-label. Then, by Theorem \ref{T-WKN}, the end vertices $v_r$ and $v_s$ of the the corresponding edge $e_j$ of $G_1$ can not have non-singleton set-label. Hence, as no copy of $G_2$ are $1$-uniform, no vertex of $G_1$ can have a non-singleton set-label. That is, $G_1$ is $1$-uniform.
\end{proof}

The converse of the theorem is also valid for with respect to the weak IASIs defined on $G_1$ and $G_2$. Let $f_1$ and $f_2$ the weak IASIs defined on $G_1$ and $G_2$, which need not be $1$-uniform.  The vertices of the copies of $G_2$ corresponding to the non-mono-indexed edges of $G_1$ need to be re-labeled using distinct singleton sets and the vertices of the copies of $G_2$ corresponding to the mono-indexed edges of $G_1$ can be labeled by distinct integral multiples of the set-labels of the corresponding vertices of $G_2$. Clearly, this new labeling will be a weak IASI of $G_1\dmd G_2$. Hence, we have the following necessary and sufficient condition for the edge corona of two weak IASI graphs to admit a weak IASI.

\begin{thm}\label{T-SNECG2}
For given weak IASI graphs $G_1$ and $G_2$, $G_1\dmd G_2$ admits a weak IASI if and only if $m_1-m_1'$ copies of $G_2$ are $1$-uniform, where $m_1'$ is the number of mono-indexed edges in $G_1$.
\end{thm}

In view of Theorem \ref{T-SNECG2}, we can estimate the number of mono-indexed edges in the edge corona of two given graphs.

\begin{thm}
For given graphs $G_1$ and $G_2$, the number of mono-indexed edges in $G_1\dmd G_2$ is $m_1'(1+m_2'+2n_2')+(m_1-m_1')(m_2+n_2)$, where $m_i$ is the number of edges and $n_i$ is the the number of vertices of $G_i$ for $i=1,2$ and $m_i'$ is the number of mono-indexed edges and $n_i'$ is the the number of mono-indexed vertices of $G_i$ with respect to a weak IASI defined on $G_i$ for $i=1,2$. 
\end{thm}
\begin{proof}
Let $G_1$ be a graph on $n_1$ vertices and $m_1$ edges and $G_2$ be a graph  on $n_2$ vertices and $m_2$ edges. Let $f_1$ and $f_2$ be the weak IASIs defined on $G_1$ and $G_2$ respectively. Let $n_i'$ and $m_i'$ be the number of vertices and edges of $G_i$ that are mono-indexed under the weak IASI $f_i$ for $i=1,2$. 

Let $G=G_1\dmd G_2$ be a weak IASI graph.  Assume that $G_1$ is not $1$-uniform. Then, $G_1$ has some elements having non-singleton set-labels. Then, by Theorem \ref{T-SNECG1}, $m_1-m_1'$ copies of $G_2$ must be $1$-uniform. Let $\mathfrak{C}_1$ be the set of all $1$-uniform copies of $G_2$ in $G_1\dmd G_2$. Therefore, the members of $\mathfrak{C}_1$ contributes a total of $(m_1-m_1')m_2$ mono-indexed edges to $G_1\dmd G_2$.

In the remaining $m_1'$ copies of $G_2$, we can label the vertices by the distinct integral multiples of the set-labels of the corresponding vertices of $G_2$ with respect to $f_2$. Let $\mathfrak{C}_2$ be the collection of these copies of $G_2$. Then, each  element in $\mathfrak{C}_2$ has $m_2'$ mono-indexed edges. Therefore, the elements of  $\mathfrak{C}_2$ contributes a total of $m_1'm_2'$ mono-indexed edges.

It remains to determine the number of mono-indexed edges between $G_1$ and different copies of $G_2$. The mono-indexed vertex of every non-mono-indexed edge of $G_1$ is adjacent to all vertices of the corresponding copy of $G_2$, which is also $i$-uniform. The number of such mono-indexed edges is $(m_1-m_1')n_2$. Both end vertices of each mono-indexed edge of $G$ are adjacent to $n_2'$ mono-indexed vertices of the corresponding copies of $G_2$. The number of such mono-indexed edges is $2m_1'n_2'$.

Therefore, the total number of mono-indexed edges in $G_1\dmd G_2$ is $m_1'+(m_1-m_1')m_2+m_1'm_2'+(m_1-m_1')n_2+2m_1'n_2'=m_1'(1+m_2'+2n_2')+(m_1-m_1')(m_2+n_2)$.
\end{proof}

In view of Theorem \ref{T-SNECG1}, let us now proceed to discuss the sparing number of the edge corona of certain graphs. We shall first consider the edge corona of two path graphs.

\begin{thm}\label{T-EC2P}
Let $P_m$ and $P_n$ be two paths on $m$ and $n$ vertices respectively, for $m,n >1$. Then, the sparing number of the edge corona of $P_m$ and $P_n$ is 
\begin{equation*}
\varphi(P_m\dmd P_n)=
\begin{cases}
\frac{1}{2}m(n+2)-1; ~~ n ~{\text is~ even}\\
\frac{1}{2}m(n+1)-1; ~~ n ~{\text is~ odd}.
\end{cases}
\end{equation*}
\end{thm} 
\begin{proof}
Let $G=P_m\dmd P_n$. Assume that an internal vertex $v$ of $P_m$ has a non-singleton set-label. Then, the $2+2n$ edges incident on $v$ become non-mono-indexed. But, two copies of $P_n$ whose vertices are adjacent to $v$ become $1$-uniform and $(n-1)$ edges of each of these copies of $P_n$ become mono-indexed. More over, $2n-1$ edges between $P_m$ and each of these two copies of $P_2$ become mono-indexed if $n$ odd and $2n-1$ edges between $P_m$ and each of these two copies of $P_2$ become mono-indexed if $n$ even. Therefore, In both cases, we have more mono-indexed edges than when $G_1$ is $1$-uniform. Therefore, $G$ has minimum number of mono-indexed edges when $G_1$ is $1$-uniform.

If $P_m$ is $1$-uniform, each copy of $P_n$ can be labeled in an injective manner alternately by non-singleton sets and singleton sets. Therefore, no edges in these copies need to be mono-indexed. Then, each vertex of $P_m$  $\lfloor \frac{n}{2} \rfloor$ mono-indexed edges together with the mono-indexed vertices of the corresponding copy of $P_n$. Therefore, if $n$ is even, $G$ has $(m-1)+m.\frac{n}{2}= \frac{1}{2}m(n+2)-1$ mono-indexed edges and if $n$ is odd, $G$ has $(m-1)+m.\frac{n-1}{2}= \frac{1}{2}m(n+1)-1$ mono-indexed edges.
\end{proof}

The sparing number of the edge corona of two graphs in which one is a path and the other is a cycle has been determined in the following theorems.

\begin{thm}\label{T-ECPC}
Let $P_m$ be a path on $m$ vertices and and $C_n$ be a cycle on $n$ vertices, for $m>1$. Then, the sparing number of the edge corona of $P_m$ and $C_n$ is 
\begin{equation*}
\varphi(P_m\dmd C_n)=
\begin{cases}
\frac{1}{2}m(n+2)-1; ~~ n ~{\text is~ even}\\
\frac{1}{2}m(n+5)-2; ~~ n ~{\text is~ odd}.
\end{cases}
\end{equation*}
\end{thm}
\begin{proof}
Let $G=P_m\dmd C_n$. As proved in Theorem \ref{T-EC2P}, $G$ has minimum number of mono-indexed edges when $P_m$ is $1$-uniform. Then, each copy of $C_n$ can be labeled in an injective manner alternately by singleton sets and non-singleton sets and hence no edges in these copies are mono-indexed. With respect to this labeling, each copy of $C_n$ contains $lceil\frac{n}{2}\rceil$ mono-indexed vertices and makes $lceil\frac{n}{2}\rceil$ mono-indexed edges with each vertex of $P_m$. Therefore, if $n$ is even, no copy of $C_n$ need to have a mono-indexed edge and hence $G$ has $(m-1)+m.\frac{n}{2}= \frac{1}{2}m(n+2)-1$ mono-indexed edges. If $n$ is odd, then each copy of $C_n$ must have a mono-indexed edge and hence $G$ has $2(m-1)+m.\frac{n+1}{2}= \frac{1}{2}m(n+5)-2$ mono-indexed edges.
\end{proof}

\begin{thm}\label{T-ECCP}
Let $C_m$ be a cycle on $m$ vertices and and $P_n$ be a path on $n$ vertices, for $n>1$. Then, the sparing number of the edge corona of $C_m$ and $P_n$ is 
\begin{equation*}
\varphi(C_m\dmd P_n)=
\begin{cases}
\frac{1}{2}m(n+2); ~~ n ~{\text is~ even}\\
\frac{1}{2}m(n+1); ~~ n ~{\text is~ odd}.
\end{cases}
\end{equation*}
\end{thm}
\begin{proof}
Let $G=P_m\dmd C_n$. As we have already proved in Theorem \ref{T-EC2P}, $G$ has minimum number of mono-indexed edges when $C_m$ is $1$-uniform. Then, each copy of $P_n$ can be labeled alternately by non-singleton sets and singleton sets and no edges in them are mono-indexed. Also, each copy of $P_n$ contains $lfloor\frac{n}{2}\rfloor$ mono-indexed vertices and makes the same number of mono-indexed edges with each vertex of $C_m$. Therefore, if $n$ is even, $G$ has $m+m.\frac{n}{2}= \frac{1}{2}m(n+2)$ mono-indexed edges and if $n$ is odd, $G$ has $m+m.\frac{n-1}{2}= \frac{1}{2}m(n+1)$ mono-indexed edges.
\end{proof}

In the following result, we study the sparing number of the edge corona of two cycle graphs. 

\begin{thm}\label{T-EC2C}
Let $C_m$ and $C_n$ be two cycles on $m$ and $n$ vertices respectively. Then, the sparing number of the edge corona of $C_m$ and $C_n$ is 
\begin{equation*}
\varphi(C_m\dmd C_n)=
\begin{cases}
\frac{1}{2}m(n+2); ~~ n ~{\text is~ even}\\
\frac{1}{2}m(n+5); ~~ n ~{\text is~ odd}.
\end{cases}
\end{equation*}
\end{thm}
\begin{proof}
Let $G=P_m\dmd C_n$. Then, as we have stated in above theorems, $G$ has minimum number of mono-indexed edges when $C_m$ is $1$-uniform and we can label each copy of $C_n$ alternately by singleton sets and non-singleton sets. Hence, no edges in these copies will be mono-indexed. With respect to this labeling, each copy of $C_n$ contains $lceil\frac{n}{2}\rceil$ mono-indexed vertices and makes $lceil\frac{n}{2}\rceil$ mono-indexed edges with each vertex of $P_m$. Therefore, if $n$ is even, $G$ has $m+m.\frac{n}{2}= \frac{1}{2}m(n+2)$ mono-indexed edges. If $n$ is odd, then each copy of $C_n$ has one mono-indexed edge and hence $G$ has $2m+m.\frac{n+1}{2}= \frac{1}{2}m(n+5)$ mono-indexed edges.
\end{proof}

So far, we have discussed about the edge corona of certain regular graphs having same vertex degree $2$. Can we generalise this result to all regular graphs having same vertex degree? The following result provide a solution to this problem.

\begin{thm}\label{T-EC2G}
Let $G_1$ and $G_2$ be two $r$-regular graphs on $m$ and $n$ vertices respectively, for $m,n >1$. Then, the sparing number of the edge corona of $G_1$ and $G_2$ is $m[n'+r(1+\varphi_2)]$, where $n'$ is the minimum number of mono-indexed vertices required in $G_2$.
\end{thm} 
\begin{proof}
Let $G=G_1\dmd G_2$. Assume that an internal vertex $v$ of $G_1$ has a non-singleton set-label. Then, the $r+2n$ edges incident on $v$ become non-mono-indexed. But, $r$ copies of $G_2$ whose vertices are adjacent to $v$ become $1$-uniform and $rn-\varphi_2$ more edges of each of these copies of $G_2$ become mono-indexed, where $\varphi_2$ is the mono-indexed edges in $G_2$. Let $n_1$ vertices having non-singleton set-labels in $G_2$ must be relabeled by singleton sets in these $r$ copies of $G_2$. Therefore, $rn_1$ edges between $G_1$ and each of these $r$ copies of $G_2$ become mono-indexed. The total number of new mono-indexed edges in $G$ is $r[rn-\varphi_2+rn_1]$. Therefore, in this case, we have more mono-indexed edges than when $G_1$ is $1$-uniform. Therefore, $G$ has minimum number of mono-indexed edges when $G_1$ is $1$-uniform.

If $G_1$ is $1$-uniform, vertices of each copy of $G_2$ can be labeled in an injective manner alternately by distinct integral multiples of the set-labels of the corresponding vertices of $G_2$. Then, the number of mono-indexed edges in each copy of $G_2$ is $varphi_2$. Then, the total number of mono-indexed edges in $G_1\dmd G_2$ is $rm+ rm\varphi_2 +mn'= m[n'+r(1+\varphi_2)]$, where $n'$ is the minimum number of mono-indexed vertices required in $G_2$. 
\end{proof}

We can extend the above theorem for the edge corona of an $r$-regular graph $G_1$ and an $s$-regular graph $G_2$ , where $r<s$ as follows.

\begin{thm}\label{T-EC2G1}
Let $G_1$ be an $r$-regular graph on $m$ and $n$ vertices and $G_2$ be an $s$-regular graph on $n$ vertices, for $m,n >1$ and $r\le s$. Then, the sparing number of the edge corona of $G_1$ and $G_2$ is $m(n'+r(1+\varphi_2))$, where $n'$ is the minimum number of mono-indexed vertices required in $G_2$.  
\end{thm}
\begin{proof}
Let $G=G_1\dmd G_2$. Assume that an internal vertex $v$ of $G_1$ has a non-singleton set-label. Then, as mentioned in the previous theorem, $r+2n$ edges incident on $v$ become non-mono-indexed. But, $r$ copies of $G_2$ corresponding to the edges incident on $v$ become $1$-uniform and hence $sn-\varphi_2$ more edges of each of these copies of $G_2$ become mono-indexed, where $\varphi_2$ is the mono-indexed edges in $G_2$. Moreover, $rn_1$ edges between $G_1$ and each of these $r$ copies of $G_2$ also become mono-indexed, where $n_1$ is the number of vertices having non-singleton set-labels in $G_2$. Hence, the number of new mono-indexed edges in $G$ is greater than  the new non-mono-indexed edges in $G$. Therefore, in this case also, we have the minimum number of mono-indexed edges when $G_1$ is $1$-uniform.

If $G_1$ is $1$-uniform, vertices of each copy of $G_2$ can be labeled in an injective manner alternately by distinct integral multiples of the set-labels of the corresponding vertices of $G_2$. Then, the number of mono-indexed edges in each copy of $G_2$ is $varphi_2$. Hence, the total number of mono-indexed edges in $G$ is $rm+ m\varphi_2 +mn'= m[n'+r+\varphi_2]$, where $n'$ is the minimum number of mono-indexed vertices required in $G_2$.
\end{proof}

Another important problem in this area is about the edge corona of two graphs in which one graph is a complete graph. First consider the edge corona of a path $P_m$ and a complete graph $K_n$. 

\begin{thm}
Let $P_m$ be a path on $m$ vertices and $K_n$ be a complete graph on $n$ vertices. Then, the sparing number of $P_m \dmd K_n$ is $\frac{1}{2}n(n+1)(m-1)$.
\end{thm}
\begin{proof}
Let $G=P_m \dmd K_n$. Then, $G$ can be considered as the one point union of $m-1$ complete graphs on $n+2$ vertices. Therefore, by Theorem \ref{T-WKN}, each $K_{n+2}$ has $\frac{1}{2}n(n+1)$ mono-indexed edges. Since each $K_{n+2}$ are edge disjoint, by Theorem \ref{T-WUG}, the total number of mono-indexed edges is $\frac{1}{2}n(n+1)(m-1)$.
\end{proof}

\ni Next, let us consider the edge corona of a cycle $C_m$ and a complete graph $K_n$.

\begin{thm}
Let $C_m$ be a cycle on $m$ vertices and $K_n$ be a complete graph on $n$ vertices. Then, the sparing number of $C_m \dmd K_n$ is $\frac{1}{2}mn(n+1)$.
\end{thm}
\begin{proof}
Let $G=C_m \dmd K_n$. Then, $G$ can be considered as the one point union of $m$ complete  graphs $K_{n+2}$ and by Theorem \ref{T-WKN}, each of these $K_{n+2}$ has $\frac{1}{2}n(n+1)$ mono-indexed edges. Since each $K_{n+2}$ are edge disjoint, by Theorem \ref{T-WUG}, the total number of mono-indexed edges is $\frac{1}{2}mn(n+1)$.
\end{proof}

The above two results can be generalised for the edge corona of an $r$-regular graph $G$ on $m$ vertices and a complete graph $K_n$ as follows.

\begin{thm}
Let $G$ be an $r$-regular graph on $m$ vertices and $K_n$ be a complete graph on $n$ vertices, where $r\le n-1$. Then, the sparing number of $G \dmd K_n$ is $\frac{1}{4}rmn(n+1)$.
\end{thm}
\begin{proof}
Since $G$ is an $r$-regular graph on $m$ vertices, then the number of edges in $G$is $\frac{1}{2}rm$. Then, $G\dmd K_n$ can be considered as a one point union of $\frac{1}{2}rm$ complete graphs on $n+2$ vertices. Then, the total number of mono-indexed edges in $G\dmd K_n$ is $\frac{1}{2}rm. \frac{1}{2}n(n+1)= \frac{1}{4}rmn(n+1)$. 
\end{proof}

\section{Conclusion}

In this paper, we have discussed about the sparing number of the edge corona of certain graphs. Some problems in this area are still open. For an $r$-regular graph $G_1$ and an $s$-regular graph $G_2$, with $r>s$, estimation of the sparing number of their edge corona is very complex. For some values $r$ and $s$, we get minimum number of mono-indexed edges when $G_1$ is $1$-uniform and for some other values of $r$ and $s$, we have minimum number of mono-indexed edges when $G_1$ is not $1$-uniform. Hence, determining the sparing number of $G_1\dmd G_2$ in such a situation still remains as an open problem.

In this paper, we have not addressed the problem of determining the sparing number of the edge corona of two graphs in which the first graph is a complete graph. The case when both the graphs are complete graphs are also not attempted.

For two arbitrary graphs, determining the sparing number of their edge corona is more complicated. The uncertainty in the adjacency and incidence pattern of arbitrary graphs makes this study complex. Hence, determining the sparing number of $G_1\dmd G_2$ for arbitrary graphs $G_1$ and $G_2$ is also an open problem.

The problems related to verifying the admissibility of weak IASIs by other graph products of two arbitrary graphs and determining the corresponding sparing numbers are also open.


\begin{thebibliography}{25}

\bibitem {BM} J. A. Bondy and U. S. R. Murty, {\bf Graph Theory}, Springer, 2008.

\bibitem{CGS1} K. P. Chithra, K. A. Germina and N. K. Sudev, {\em The Sparing Number of the Cartesian product of Certain Graphs}, Communications in Mathematics \& Applications, {\bf 5}(1)(2014), 23-30.

\bibitem{CGS2} K. P. Chithra, K. A. Germina and N. K. Sudev, {\em A Study on the Sparing Number of the Corona of Certain Graphs}, Research \& Reviews: Discrete Mathematical Structures, {\bf 1}(2)(2014), 5-15.

\bibitem{FFH} R. Frucht and F. Harary, {\em On the Corona of Two Graphs}, Aequationes Math., {\bf 4}(3)(1970), 322-325.

\bibitem{JAG} J. A. Gallian, {\em A Dynamic Survey of Graph Labelling}, The Electronic Journal of Combinatorics, \# DS 16, 2013.

\bibitem{GA} K. A. Germina and T. M. K. Anandavally, {\em Integer Additive Set-Indexers of a Graph:Sum Square Graphs}, Journal of Combinatorics, Information and System Sciences, {\bf 37}(2-4)(2012), 345-358.

\bibitem{GS1} K. A. Germina and N K Sudev, {\em On Weakly Uniform Integer Additive Set-Indexers of Graphs}, International Mathematical Forum, {\bf 8}(37)(2013), 1827-1834.

\bibitem{HIS} R. Hammack, W. Imrich and S. Klavzar, {\bf Handbook of Product graphs}, CRC Press, 2011.

\bibitem{FH}  F. Harary, {\bf Graph Theory}, Addison-Wesley Publishing Company Inc., 1994.

\bibitem{HS} Y. Hu and W. C. Shiu, {\em The spectrum of the edge corona of two graphs}, Electronic Journal of Linear Algebra, {\bf 20}(2010), 586-594.

\bibitem {GS3} N. K. Sudev and K. A. Germina, {\em A Characterisation of Weak Integer Additive Set-Indexers of Graphs}, Journal of Fuzzy Set Valued Analysis, {\bf 2014}(2014), Article Id: jfsva-0189, $7$ pages.

\bibitem {GS4} N. K. Sudev and K. A. Germina, {\em Weak Integer Additive Set-Indexers of Graph Operations}, Global Journal of Mathematical Sciences: Theory and Practical, {\bf 6}(1)(2014),25-36.

\bibitem{GS5} N. K. Sudev and K. A. Germina, {\em Further Studies on the Sparing Number of Graphs}, TechS Vidya e-Journal of Research, {\bf 2}(2013-14), 28-38.

\bibitem{GS6} N.K. Sudev and K. A. Germina, {\em A Note on Sparing Number of Graphs}, Advances and Applications in Discrete Mathematics, {\bf 14}(1)(2014), 51-65.

\bibitem{GS7} N.K. Sudev and K. A. Germina, {\em On Weak Integer Additive Set-Indexers of Certain Graph Classes}, Journal of Discrete Mathematical Sciences and Cryptography, to appear.

\bibitem{GS8} N. K. Sudev and K. A. Germina, {\em Weak integer Additive Set-Indexers of Certain Graph  Products}, Journal of Informatics and Mathematical Sciences, {\bf 6}(1)(2014), 35-43.

\bibitem{DBW} D B West, (2001). {\em Introduction to Graph Theory}, Pearson Education Inc.

\end{thebibliography}
\end{document}